\documentclass[twoside,a4paper,12pt]{article}
\usepackage[top=4.5cm, bottom=5cm, inner=3.5cm,outer=3.5cm]{geometry}
\usepackage[utf8]{inputenc}
\usepackage{amsthm}
\usepackage{mathtools}
\usepackage{csquotes}
\usepackage{amsmath}
\usepackage{graphicx}
\usepackage{dsfont}
\usepackage{soul}
\usepackage{amssymb}
\usepackage{multicol}
\usepackage{courier}
\usepackage{listings}
\usepackage{eurosym}
\usepackage{subfiles}
\usepackage{multicol}
\usepackage{pdfpages}
\usepackage{footnote}
\usepackage{color}
\definecolor{myblue}{rgb}{.25, .25, .9}
\definecolor{myred}{rgb}{.6, .3, .3}
\definecolor{myred2}{rgb}{.6, .1, .1}
\definecolor{mygreen}{rgb}{.25, .6, .5}
\usepackage{adjustbox}
\usepackage[colorlinks=true,
            linkcolor=myred2,
            urlcolor=myred,
            citecolor=myred]{hyperref}
\usepackage{pdfpages}
\usepackage[english]{babel}
\usepackage{array}
\usepackage{enumerate}
\usepackage{gensymb}
\usepackage[nottoc]{tocbibind}
\usepackage{lettrine}
\usepackage[normalem]{ulem}
\usepackage[font=small,labelfont=bf]{caption}
\usepackage{cancel}
\usepackage{nomencl}
\usepackage{mathrsfs}
\usepackage{changepage}
\usepackage{amsmath,stmaryrd,graphicx,float}
\numberwithin{equation}{section}

\usepackage{tikz-cd}
\usetikzlibrary{arrows}

\usepackage{subcaption} 

\usepackage{titlesec}

\newtheoremstyle{mystyle}%                % Name
  {}%                                     % Space above
  {}%                                     % Space below
  {\itshape}%                             % Body font
  {}%                                     % Indent amount
  {\bfseries}%                            % Theorem head font
  {.}%                                    % Punctuation after theorem head
  { }%                                    % Space after theorem head, ' ', or \newline
  {}%                                     % Theorem head spec (can be left empty, meaning `normal')

\theoremstyle{mystyle}

\newtheorem{theorem}{Theorem}[section]

\newtheorem{lemma}[theorem]{Lemma}
\newtheorem{proposition}[theorem]{Proposition}
\newtheorem{corollary}[theorem]{Corollary}

\newtheorem{remark}[theorem]{Remark} 
\newtheorem{example}[theorem]{Example}

\newcommand\xqed[1]{%
  \leavevmode\unskip\penalty9999 \hbox{}\nobreak\hfill
  \quad\hbox{#1}}
\newcommand\exampleEnd{\xqed{$\circ$}}

\usepackage{etoolbox}
\DeclareMathAlphabet{\mathbbold}{U}{bbold}{m}{n}

\titleformat{\section}[runin]
{\bfseries}{\llap{\thesection\hskip 9pt}}{0pt}{}
\titlespacing*{\section}
{0pt}{3.25ex plus 1ex minus .2ex}{1.5ex plus .2ex}

\titleformat{\subsection}[runin]
{\bfseries}{\llap{\thesubsection\hskip 9pt}}{0pt}{}
\titlespacing*{\subsection}
{0pt}{3.25ex plus 1ex minus .2ex}{1.5ex plus .2ex}

\titleformat{\subsubsection}[runin]
{\bfseries}{\llap{\thesubsubsection\hskip 9pt}}{0pt}{}
\titlespacing*{\subsubsection}
{0pt}{3.25ex plus 1ex minus .2ex}{1.5ex plus .2ex}

\titleformat{\paragraph}[runin]
{\bfseries}{\llap{\theparagraph\hskip 9pt}}{0pt}{}
\titlespacing*{\paragraph}
{0pt}{3.25ex plus 1ex minus .2ex}{1.5ex plus .2ex}

\usepackage[titles]{tocloft}
\usepackage{fancyhdr}
\usepackage{lastpage}
\fancypagestyle{plain}{
\fancyhf{}
\fancyhead[LE]{\thepage\quad\textbf{\nouppercase\leftmark}}
\fancyhead[RO]{\textbf{\nouppercase\rightmark}\quad\thepage}
}

\fancypagestyle{basicstyle}{%
\fancyhf{}
\fancyhead[LE]{\thepage}
\fancyhead[RO]{\thepage}
}

\usepackage{algorithm}
\usepackage{algorithmic}[1]

\usepackage{chngcntr}
\counterwithin{figure}{section}

\usepackage[backend=bibtex,
doi=false,
maxbibnames=99,
style=alphabetic,
bibencoding=ascii,
giveninits=true, 
]{biblatex}
\renewbibmacro{in:}{}

\usepackage{rotating}
\usepackage{mathtools}

\title{{\textbf{\Large On global Lyapunov functions being Morse}}}
\author{Wouter Jongeneel\footnote{The author is with the KTH Royal Institute of Technology and Digital Futures, Stockholm, Sweden. Contact: \texttt{wouterjo}\texttt{@kth.se}, \texttt{wjongeneel.nl}. This work received funding from Digital Futures through the project ``\textit{Programmability of Cells}". The author would like to thank Matthew D. Kvalheim for feedback, inspiring discussions and for generously suggesting Proposition~\ref{prop:counter}. He would also like to thank Christopher Criscitiello for bringing \cite{ref:boumal2026smooth} to his attention.}
}
\date{\small{\today}}

\begin{document}
\maketitle
\thispagestyle{empty}

\begin{abstract}
Lyapunov functions have been a cornerstone of dynamical systems theory ever since the early 1900s. Moreover and akin to Morse theory, Lyapunov functions have been the key in linking topology to dynamical systems theory. In this note we continue along these lines. Suppose that some vector field on Euclidean space has a unique equilibrium point that is globally asymptotically stable (GAS). Under these conditions, it is known that there is always a smooth Lyapunov function to certify stability. One may wonder if this function can always be chosen to be Morse. We show that if the equilibrium point is hyperbolic, this is indeed true. Moreover, we show that if this would be true in general, then the generalized $4$-dimensional smooth Poincar\'e conjecture (SPC4) must be true. Even so, we provide an example that is GAS, but does not admit a Morse Lyapunov function, closing this route of proving SPC4. 
\end{abstract}

{\footnotesize{
}}

% 55P10  	Homotopy equivalences in algebraic topology
% 93D20  	Asymptotic stability in control theory
% 	37B25  	Stability of topological dynamical systems

 \section{Introduction}
Starting with Massera~\cite{ref:Massera1956}, Kurzweil~\cite{ref:kurzweil1963inversion} and Wilson \cite{ref:wilson1969smoothing}, there is now a vast theory describing when continuous or smooth Lyapunov functions exist to capture the existence of a stable attractor under some dynamical system, \textit{e.g.}, see~\cite{hahn1967stability,ref:bhatia1970stability,ref:bacciotti2005liapunov,ref:fathi2019smoothing}. Ideally, however, one can say more beyond regularity of those functions. For instance, to efficiently and effectively work with parametric function spaces when searching for a Lyapunov function. However, arguably more interesting, a better structural understanding of Lyapunov functions connects to the generalized Poincar\'e conjecture (\textit{e.g.}, see \cite{ref:hirsch1965homotopy,ref:wilson1967structure} or more recently \cite{ref:JongeneelSCL26,ref:kvalheim2025global,ref:kvalheim2025differential}), as further discussed in this note.

\subsection{Setting}
Given a $C^r$-smooth dynamical system 
\begin{equation}
\label{equ:f}
    \dot{x}=f(x), \quad x\in \mathbb{R}^n,
\end{equation}
for some $r\in \mathbb{N}_{\geq 1}\cup \{\infty\}$, and suppose that some equilibrium point $x^*\in \mathbb{R}^n$ is \textit{globally asymptotically stable}\footnote{For the reader unfamiliar with stability, we briefly define GAS. Consider~\eqref{equ:f}, and let the corresponding (semi)flow be denoted by $\varphi$, \textit{i.e.}, $\frac{\mathrm{d}}{\mathrm{d}t}\varphi^t(x)|_{t=s} = f(\varphi^s(x))$ for all $(x,s)\in \mathbb{R}^n\times \mathbb{R}_{\geq 0}$. A point $x^{*}\in \mathbb{R}^n$ is an \textit{equilibrium point} of $f$ when $f(x^{*})=0$ or equivalently $\varphi^s(x^{*})=x^{*}$ $\forall s\in \mathbb{R}_{\geq 0}$. Then, the equilibrium point $x^*$ is said to be \textit{{globally asymptotically stable}} (GAS) (under $f$) when: $(i)$ $x^*$ is \textit{Lyapunov stable}, that is, for any open neighbourhood $U_{\varepsilon}\ni x^*$ there is an open set $U_{\delta}\subseteq U_{\varepsilon}$ such that $\varphi^s(U_{\delta})\subseteq U_{\varepsilon}$ $\forall s\in \mathbb{R}_{\geq 0}$; and $(ii)$ $x^*$ is \textit{globally attractive}, that is, $\lim_{s\to+\infty}\varphi^s(x_0)=x^*$ for all $x_0\in \mathbb{R}^n$.} (GAS) (\textit{e.g.}, see \cite[p. 28]{ref:bacciotti2005liapunov}). This implies that there is a smooth (\textit{i.e.}, $C^{\infty}$-smooth) function $V:\mathbb{R}^n\to \mathbb{R}_{\geq 0}$ that is $(i)$ coercive\footnote{A function $V:\mathbb{R}^n\to \mathbb{R}_{\geq 0}$ is said to be \textit{coercive} when $\{x\in \mathbb{R}^n:V(x)\leq c\}$ is compact for all $c\in \mathbb{R}_{\geq 0}$. Coercivity implies in particular that $\lim_{\|x\|\to +\infty}V(x)=+\infty$.}; $(ii)$ such that $x^*=V^{-1}(0)$; and $(iii)$ away from $x^*$, decreasing strictly under solutions of~\eqref{equ:f}, \textit{i.e.}, $x^*=\nabla V^{-1}(0)$ and  $\langle \nabla V(x),f(x) \rangle < 0$ for all $x\in \mathbb{R}^n\setminus \{x^*\}$, \textit{e.g.}, see \cite[Thm~2.4]{ref:bacciotti2005liapunov}. Functions $V$ satisfying properties $(i)$, $(ii)$ and $(iii)$ will be simply called smooth \textbf{\textit{Lyapunov functions}} (this instead of smooth, strict, proper Lyapunov functions). Without loss of generality, in what follows, we let $x^*=0$. At last, we emphasize that vector fields on $\mathbb{R}^n$ can be identified with maps from $\mathbb{R}^n$ to itself as $T\mathbb{R}^n\cong \mathbb{R}^n\times \mathbb{R}^n$.

\subsection{Structure of Lyapunov functions}
Beyond regularity, based on knowledge of the underlying dynamical system, one is sometimes able to state stronger converse results.

Here, we recall that a function $g\in C^2(\mathbb{R}^n;\mathbb{R})$ is a \textbf{\textit{Morse function}} when all its critical points (\textit{i.e.}, $x^*\in \mathbb{R}^n$ such that $\nabla g(x^*)=0$) are non-degenerate, that is, the Hessian $\nabla^2 g(x^*)$ is non-singular. Then, inspired by Morse theory (\textit{e.g.}, see \cite{ref:milnor1963morse}), one may consider the study, \textit{globally}, of Lyapunov functions that have a positive definite Hessian at $x^*$, \textit{i.e.}, $\langle u,\nabla^2V(x^*) u\rangle > 0$ for all $u\in \mathbb{R}^n\setminus \{0\}$, compactly written as $\nabla^2 V(x^*)\succ 0$. Such a Lyapunov function will be called a \textit{\textbf{Morse Lyapunov function}}. We are partially inspired to do so by the following observations. 

First, when~\eqref{equ:f} is linear, $V$ can be chosen to be a \textit{quadratic} form, \textit{e.g.}, see \cite[Thm.~18]{ref:sontag2013mathematical}. Then, Taylor series arguments motivate the use of quadratic Lyapunov functions for nonlinear systems, locally, \textit{e.g.}, see \cite[\S 9]{ref:lefschetzlasalle1961}, \cite[\S 28]{hahn1967stability} or more recent~\cite{ref:barreau2024supervised,ref:harting2026locally}.   

Beyond linear systems and quadratic Lyapunov functions, little structural knowledge is available. For instance, one may hope that \textit{polynomial} systems theory is closed in some sense. Unfortunately, counterexamples exist. For instance, it is known that despite $0\in \mathbb{R}^2$ being GAS under the planar polynomial dynamical system $\dot{x}=f_p(x)$, defined through $(x_1,x_2)=:x \mapsto f_p(x):= (-x_1+x_1x_2,-x_2)$, there is no polynomial Lyapunov function to assert this~\cite{ref:ahmadi2011globally}. A \textit{non}-polynomial Lyapunov function, however, is given by $x\mapsto V_{np}(x):=\log(1+x_1^2)+x_2^2$. As $Df_p(0)$ is Hurwitz, one does expect to find a locally quadratic Lyapunov function and indeed, $V_{np}$ is a Morse function as $\nabla^2 V_{np}(0)=2\cdot I_2\succ 0$. Similarly, one can construct vector fields on $\mathbb{R}^2$ that render the origin GAS, but provably fail to admit a \textit{convex} Lyapunov function \cite[Ex. III.3]{ref:JongeneelSchwan2024TAC}.
% we have 
% \begin{align*}
%     \nabla V_{np}(x_1,x_2) = \begin{pmatrix}
% \dfrac{2x_1}{(1+x_1^2)^2}\\
% 2x_2
%     \end{pmatrix}, \quad \nabla^2 V_{np}(x_1,x_2) = \begin{pmatrix}
%         \dfrac{2}{(1+x_1^2)^2} - \dfrac{4x_1^2}{(1+x_1^2)^3}  & 0\\
%         0 & 2
%     \end{pmatrix}
% \end{align*}
% and hence $(0,0)\in \mathbb{R}^2$ is a non-degenerate critical point.

% Note that
% \begin{align*}
%     Df_p(x) = \begin{pmatrix}
%         -1 + x_2 & x_1\\
%         0 & -1
%     \end{pmatrix},
% \end{align*}
% such that $(0,0)$ is a hyperbolic equilibrium point, meaning that locally, a quadratic Lyapunov function must exist.
Next, one may consider the real-analytic setting. However, in \cite[Prop. 5.2 (b)]{ref:bacciotti2005liapunov} one finds again a planar vector field that renders $0$ GAS but fails to admit a locally real-analytic Lyapunov function, see also \cite{ref:ahmadi2018globallyP}. Those results by themselves, however, do not obstruct the existence of a Morse Lyapunov function. Consider, for example, the function $\mathbb{R}\ni x\mapsto v(x):=x^2+\mathrm{exp}(-1/x^2)$. This function is smooth, coercive and Morse. However, this function is not real-analytic at $0$. One may remark that for both \cite{ref:bacciotti2005liapunov} and \cite{ref:ahmadi2018globallyP}, their examples are such that $0\in \mathbb{R}^2$ is a non-hyperbolic equilibrium point. Again, a lack of hyperbolicity appears to be no immediate obstruction to finding Morse Lyapunov functions, \textit{e.g.}, for the vector field $\mathbb{R}\ni x\mapsto f(x):=-x^3$ one simply picks $\mathbb{R}\ni x\mapsto V(x):=x^2$ as a valid Lyapunov function.
% Here, a particularly well-studied (\textit{e.g.}, due to \cite{ref:Fazel_18}) subclass are functions that satisfy the so-called {\textit{global P{\L} inequality}}, that is, all $V:\mathbb{R}^n\to \mathbb{R}_{\geq 0}$ that are $C^r$, with $r\in \mathbb{N}_{\geq 1}\cup \{\infty\}$, and such that for some $\mu>0$
% \begin{equation}
%     \label{equ:PL:ineq}
%     V(x) - \inf_{x\in \mathbb{R}^n}V(x) \leq \frac{1}{2\mu}\| \nabla V(x) \|_2^2,\quad \forall x\in \mathbb{R}^n.
% \end{equation}
% We call such a function simply \textbf{\textit{globally P{\L}}}. Now, one may wonder: ``\textit{If some equilibrium point $x^*$ is GAS under a dynamical system ~\eqref{equ:f}, is there always a smooth Lyapunov function---satisfying the global PL inequality \eqref{equ:PL:ineq}---to assert this?}" 
% \\\\
% Note that a lack of hyperbolicity is no immediate obstruction, \textit{e.g.}, for $f(x):=-x^3$ one simply picks $V(x):=\frac{1}{2}x^2$, which is globally P{\L} with $\mu=1$.  

% Interestingly, $V_{np}$ does not satisfy the global P{\L} inequality (\textit{e.g.}, see that $\lim_{|x|\to \infty} \nabla V_{np}(x)=0$), however, $\nabla^2 V_{np}(0)=2\cdot I_2 \succ 0$ does hold. 
Then, inspired by the above, we ask: 
\\\\
\textbf{Question:} ``\textit{If some equilibrium point $x^*\in \mathbb{R}^n$ is GAS under a dynamical system ~\eqref{equ:f}, is there always a smooth Lyapunov function---with a positive definite Hessian at $x^*$ (\textit{i.e.}, a Morse Lyapunov function)---to assert this?}" 
\\\\
To the best of our knowledge, this question is still open. In this note we will partially answer it, \textit{e.g.}, see Theorem~\ref{thm:Morse:exists:hyperbolic}, Proposition~\ref{prop:counter} and Corollary~\ref{cor:no:Morse:exists}. Before doing so, the next section provides additional motivation for studying this question via a connection with topology. 

We also mention that these \textit{\text{Morse Lyapunov functions}} were studied explicitly before. For instance, Moulay observed that if one works with a Morse Lyapunov function, then, \textit{locally} one is effectively working with a quadratic function indeed \cite[Cor. 9]{ref:moulay2011morse}. We remark that Lyapunov functions in the context of Conley's fundamental theorem of dynamical systems are also sometimes called Morse Lyapunov functions as their Lyapunov function helps encoding a Morse decomposition, \textit{e.g.}, see \cite{ref:li2012smooth}. 

\section{Implications of Lyapunov functions being Morse}
To study this question, we first follow: (i) recent topological advances regarding globally P{\L}\footnote{Within the area of optimization, a particularly well-studied (\textit{e.g.}, due to \cite{ref:Fazel_18}) subclass of Morse functions are functions that satisfy the so-called {\textit{global Polyak-{\L}ojasiewicz (P{\L}) inequality}}, that is, all $V:\mathbb{R}^n\to \mathbb{R}_{\geq 0}$ that are $C^r$, with $r\in \mathbb{N}_{\geq 1}\cup \{\infty\}$, and such that for some $\mu>0$ the inequality $V(x) - \inf_{x\in \mathbb{R}^n}V(x) \leq \frac{1}{2\mu}\| \nabla V(x) \|_2^2$ holds for all $x\in \mathbb{R}^n.$ Such a function will be simply called {\textit{globally P{\L}}}.} (and Morse) functions by Boumal, Criscitiello and Rebjock \cite{ref:boumal2026smooth} and; (ii) recent topological insights into the space of stable dynamical systems by Kvalheim and the author \cite{ref:JongeneelSchwan2024TAC,ref:JongeneelSCL26,ref:kvalheim2025differential,ref:Kvalheim2026Smale}, building upon earlier work by Wilson~\cite{ref:wilson1967structure} and Gr\"une, Wirth and Sontag \cite{ref:grune1999asymptotic}. 

We recall that the {\textit{(generalized) $4$-dimensional smooth Poincar\'e conjecture}} (also referred to as SPC4) states that every $4$-dimensional smooth homotopy sphere\footnote{For readers unfamiliar with the notion of homotopy, we clarify what a homotopy sphere is. Two continuous maps $f,f':X\to Y$, over the \textit{topological spaces} $X$ and $Y$, are \textit{homotopic}, denoted $f\simeq_h f'$, when there is a continuous map $H:X\times [0,1]\to Y$ such that $H(\cdot,0)=f$ and $H(\cdot,1)=f'$. Then, $X$ and $Y$ are {\textit{homotopy equivalent}} when there are continuous maps $f:X\to Y$ and $g:Y\to X$ such that $f\circ g\simeq_h \mathrm{id}_Y$ and $g\circ f\simeq_h\mathrm{id}_X$. Now, a topological space is a homotopy sphere when it is homotopy equivalent to a standard $n$-sphere. For reference, see~\cite[Ch. 0]{Hatcher}.} is diffeomorphic to the standard $4$-sphere $\mathbb{S}^4$. We also recall that, in general, the smooth  Poincar\'e conjecture is false, \textit{e.g.}, as is known for $n=7$ \cite{ref:Milnor7}. The topological case, however, is known to be true, as shown by Freedman~\cite[Thm. 1.6]{ref:freedman1982topology}, see also \cite{ref:behrens2021disc}. 

Now, define
\begin{equation*}
    \mathscr{V} := \big\{ V\in C^{\infty}(\mathbb{R}^n;\mathbb{R}_{\geq 0})\,|\, \exists\, \varphi\in \mathrm{Diff}(\mathbb{R}^n):V(\varphi^{-1}(y))=\inf_{y\in \mathbb{R}^n} V(y)+\|y\|_2^2\big\}. 
\end{equation*}

It turns out that most Lyapunov functions as discussed above (\textit{e.g.}, quadratic, globally P\L, Morse) are elements of $\mathscr{V}$~\cite[Thm. 3.3]{ref:boumal2026smooth}. We remark that, compared to standard Morse theory, the diffeomorphisms employed in $\mathscr{V}$ are global. We also remark that $\mathscr{V}$ connects to templates used in the context of learning stability certificates, \textit{e.g.}, see \cite{ref:rana2020euclideanizing}. Hence, understanding expressiveness of $\mathscr{V}$ is of interest. 

Next, following \cite{ref:kvalheim2025differential}, let $\mathcal{S}^r_0(\mathbb{R}^n)$ be the set of $C^r$-smooth vector fields on $\mathbb{R}^n$, for some $r\in \mathbb{N}_{\geq 1}\cup \{\infty\}$, such that $0$ is GAS. We endow $\mathcal{S}^r_0(\mathbb{R}^n)$ with the compact-open $C^r$-topology, see \cite{ref:fox1945topologies}, \cite[Ch.2]{ref:hirsch1976differential}. Kvalheim has shown that $\mathcal{S}^r_0(\mathbb{R}^n)$ is path-connected for $n\neq 4,5$ \cite[Cor. 7]{ref:kvalheim2025differential}. 

\begin{proposition}[On Lyapunov functions being quadratic, up to a change of coordinates]
\label{prop:1}
    Let $f\in \mathcal{S}^r_0(\mathbb{R}^n)$, for some $r\in \mathbb{N}_{\geq 1}\cup \{\infty\}$ and $n\in \mathbb{N}_{\geq 1}$. If the fact that $0$ is GAS under $f$ can always be certified by a Lyapunov function $V\in \mathscr{V}$, then 
    \begin{enumerate}[(i)]
        \item $\mathcal{S}^r_0(\mathbb{R}^n)$ is path-connected and;
        \item for $n=5$, the validity of the generalized $4$-dimensional smooth Poincar\'e conjecture is implied. 
    \end{enumerate}
\end{proposition}
\begin{proof}

Regarding $(i)$, suppose we can indeed find a Lyapunov function $V\in \mathscr{V}$, then by its definition there is diffeomorphism $\varphi\in\mathrm{Diff}(\mathbb{R}^n)$ from $\mathbb{R}^n$ to itself such that $V(\varphi^{-1}(y))=\|y\|_2^2$ (as $\inf_{y\in \mathbb{R}^n} V(y)=0$). Thanks to the symmetry of $y\mapsto \|y\|_2^2$, one may pick $\varphi\in \mathrm{Diff}^+(\mathbb{R}^n)$ (\textit{i.e.}, select $\varphi$ as being an orientation preserving diffeomorphism), \textit{e.g.}, compose with a reflection $y_i\mapsto -y_i$ if needed. However, then, $\varphi$ is smoothly isotopic to $y\mapsto \mathrm{id}_{\mathbb{R}^n}(y):=y$ \cite[p. 34]{ref:milnor65}. Let $[0,1]\ni s\mapsto \psi_s$ describe this isotopy, from $\mathrm{id}_{\mathbb{R}^n}$, to $\varphi$. One may select the isotopy such that $\psi_s(0)=0$ $\forall s\in [0,1]$. Let $f_s:=(\psi_s)_* f$ (\textit{i.e.}, $f_s$ is the pushforward of $f$ under $\psi_s$ \cite[p. 183]{Lee2}), then as $f=f_0\in \mathcal{S}^r_0(\mathbb{R}^n)$ and $\psi_s$ is a smooth change of coordinates, $f_s\in \mathcal{S}^r_0(\mathbb{R}^n)$ $\forall s\in [0,1]$. 
%Then, consider the vector field $f_s$, defined through 
%\begin{align*}
%    \mathbb{R}^n\ni y \mapsto f_s(y):=D(\psi_s^{-1})(\psi_s(y))f(\psi_s(y)).
%\end{align*}
%Now, as $\psi_s$ is a smooth change of coordinates, one observes that $0$ is GAS under $f_s$, asserted by the Lyapunov function $V_s(y):=V(\psi_s(y))$, \textit{e.g.}, to be more explicit, as $D(\psi_s^{-1})(\psi_s(y))D\psi_s(y)=I_n$, we have $\langle \nabla V_s(y),f_s(y)\rangle = \langle \nabla V (\psi_s(y)),f(\psi_s(y))\rangle$ and thus global asymptotic stability throughout the isotopy. 
Hence, $[0,1]\ni s \mapsto f_s\in \mathcal{S}^r_0(\mathbb{R}^n)$ is a smooth path, from $f$ (equipped with Lyapunov function $V=:V_0$), to $f_1=\varphi_* f$ (equipped with a trivial quadratic Lyapunov function $y\mapsto V_1(y):=\|y\|_2^2$). Now one constructs a straight-line homotopy from $f_1$ to the canonical vector field $f_c\in \mathcal{S}^r_0(\mathbb{R}^n)$ defined through $y\mapsto f_c(y):=-y$. To construct this homotopy, observe that $-\nabla V_1(y)/2 = -y$, such that $\left\langle \nabla V_1(y), (1-\tau)f_1(y)-\tau \nabla V_1(y)/2 \right\rangle < 0$ for all $(\tau,y) \in [0,1]\times \mathbb{R}^n\setminus \{0\}$. As this can be done for any $f$, one concludes by transitivity.    

Regarding $(ii)$, if $\mathcal{S}^r_0(\mathbb{R}^5)$ is path-connected, for some $r\in \mathbb{N}_{\geq 1}\cup \{\infty\}$, then the smooth $4$-dimensional Poincar\'e conjecture is true \cite[Prop. 2]{ref:kvalheim2025differential}.
\end{proof}

\begin{remark}[Beyond $n=5$]
\upshape{
    Proposition~\ref{prop:1} $(ii)$ relies on a result by Hirsh stating the every homotopy $4$-sphere $S^4$ is diffeomorphic to the levelset $h^{-1}(c)$ of some smooth function $h:\mathbb{S}^5\to [0,1]$, having just two critical points \cite[Thm. 2]{ref:hirsch1965homotopy}. This means that the arguments above do not readily generalize beyond $\mathbb{R}^5$. 
    }
    \exampleEnd
\end{remark}

% As alluded to before, it is of interest to consider Lyapunov functions with a positive definite Hessian at the equilibrium instead of focussing on functions that are quadratic, polynomial or real-analytic functions. We call such a function a .  

% \begin{corollary}[Morse Lyapunov functions imply the generalized smooth Poincar\'e conjecture]
% \label{cor:Morse}
%     Let $f\in \mathcal{S}^r_0(\mathbb{R}^5)$, for some $r\in \mathbb{N}_{\geq 1}\cup \{\infty\}$. If the fact that $0$ is GAS under $f$ can always be asserted by a Morse Lyapunov function (\textit{e.g.}, a globally P{\L} function), then, the validity of the generalized $4$-dimensional smooth Poincar\'e conjecture is implied. 
% \end{corollary}
% \begin{proof}
%     By \cite[Thm. 3.3]{ref:boumal2026smooth}, for smooth Lyapunov functions with a positive definite Hessian at $x^*=0$, it is known that there is a diffeomorphism $\varphi$ such that $V(\varphi^{-1}(y))=\|y\|_2^2$. Indeed, functions that are globally P{\L} are a subclass, \textit{e.g.}, see \cite[Lem. 2.2]{ref:boumal2026smooth}. 
% \end{proof}

Although we believe that elucidating the link between the structure of Lyapunov functions and the topology of $\mathcal{S}^r_0(\mathbb{R}^n)$ is of independent interest, one may prove Proposition~\ref{prop:1} $(ii)$ directly via the aforementioned result of Hirsch \cite[Thm. 2]{ref:hirsch1965homotopy} and the techniques to prove \cite[Thm. 3.1]{ref:milnor1963morse}.

\section{Existence results}
Inspired by genericity results available for Morse functions (\textit{e.g.}, see \cite{ref:peixoto1967approximation}, \cite[p. 147]{ref:hirsch1976differential} or \cite[p. 153]{palis1982geometric}), one might expect an affirmative answer to our question. 

 To add, following~\cite{ref:kvalheim2025differential}, let $\mathcal{L}^r_0(\mathbb{R}^n)$ be the set of surjective, proper $C^r$-smooth functions $L:\mathbb{R}^n\to \mathbb{R}_{\geq 0}$ with a unique minimum and critical point at $0$ such that $L(0)=0$. This set may be understood as a set of Lyapunov functions. Now it turns out that path-connectedness of merely $\mathcal{L}^r_0(\mathbb{R}^5)$, for some $r\in \mathbb{N}_{\geq 1}\cup \{\infty\}$, also implies the generalized $4$-dimensional smooth Poincar\'e conjecture \cite[Prop. 2]{ref:kvalheim2025differential}. Hence, suppose that $f$ is of the form $f=-\nabla L$, for $L\in \mathcal{L}_0^{\infty}(\mathbb{R}^5)$. If any such a $f$ comes equipped with a Morse Lyapunov function (\textit{e.g.}, via a perturbation), then the generalized $4$-dimensional smooth Poincar\'e conjecture is true (by Proposition~\ref{prop:1}). Earlier work by Smale points in this direction (for closed manifolds) \cite{ref:smale1961gradientDS}.

Towards addressing our question, and in line with aforementioned genericity results, we start by providing a converse result. To that end, let $\mathcal{HS}_0^r(\mathbb{R}^n)\subset \mathcal{S}_0^r(\mathbb{R}^n)$ be the subset of vector fields such that $0\in \mathbb{R}^n$ is a hyperbolic equilibrium point.  

\begin{theorem}[On the existence of Morse Lyapunov functions]
\label{thm:Morse:exists:hyperbolic}
  Let $f\in \mathcal{HS}^r_0(\mathbb{R}^n)$, for some $r\in \mathbb{N}_{\geq 1}\cup \{\infty\}$ and $n\in \mathbb{N}_{\geq 1}$, then, there is a smooth Morse Lyapunov function certifying that $0$ is GAS. 
\end{theorem}

\begin{proof}
    First, by the GAS property, there is a smooth (\textit{i.e.}, $C^{\infty}$-smooth) Lyapunov function $V_{\infty}: \mathbb{R}^n\to \mathbb{R}_{\geq 0}$, \textit{e.g.}, see \cite[Thm~2.4]{ref:bacciotti2005liapunov}. This function is not known to be Morse. 

    Second, by hyperbolicity, there is some open neighbourhood $B_0$ of $0$ and a quadratic Lyapunov function $V_q:B_0\to \mathbb{R}_{\geq 0}$, such that $\langle \nabla V_q(x),f(x)\rangle < 0$ on $B_0\setminus \{0\}$, \textit{e.g.}, see~\cite[\S 28]{hahn1967stability}. This means in particular that $\nabla^2 V_q(0)\succ 0$. Note, $V_q$ need not be a globally valid Lyapunov function.

    Now we combine $V_q$ and $V_{\infty}$ to create a global, smooth Lyapunov function that is Morse. Pick $c>0$ such that $M_2:=V_{\infty}^{-1}([0,c])\subset B_0$ and set $M_1:=V_{\infty}^{-1}([0,c/2])$. Then, define $V_M:\mathbb{R}^n\to \mathbb{R}_{\geq 0}$ through
\begin{equation*}
    V_M(x) := \begin{cases}
        \varepsilon V_q(x) \quad &x\in \mathrm{int}(M_1)\\
        (1-\beta(x))\varepsilon V_q(x) + \beta(x) V_{\infty}(x)\quad &x\in \mathrm{cl}(M_2\setminus M_1)\\
        V_{\infty}(x) \quad &x\in \mathbb{R}^n\setminus M_2
    \end{cases},
\end{equation*}
for some smooth $\beta:\mathrm{cl}(M_2\setminus M_1)\to [0,1]$ and $\varepsilon>0$. We start by constructing an appropriate $\beta$. Define the function $\psi:\mathrm{cl}(M_2\setminus M_1)\to [0,1]$ through 
\begin{equation*}
     \psi(x) := \frac{V_{\infty}(x)-c/2}{c-c/2},
\end{equation*}
then consider a standard smooth bump function $b:[0,1]\to [0,1]$ defined as 
\begin{equation*}
    b(s) := \frac{\phi(s)}{\phi(s)+\phi(1-s)}
\end{equation*}
for $\phi:[0,1]\to[0,1]$ being a smooth function defined by
\begin{align*}
    \phi(s) := \begin{cases}
        e^{-\frac{1}{s}},\quad &s>0\\
        0 \quad &s=0
    \end{cases}.
\end{align*}
Now set $\beta := b\circ \psi$. 
Note, we employed standard bump functions here, that is, $b$ is smooth and such that $b(0)=0$, $b(1)=1$ with $b^{(k)}(0)=b^{(k)}(1)=0$ for all $k>0$, \textit{e.g.}, see \cite[Lem. 2.20]{Lee2}. Useful for the analysis below is that $b^{(1)}(s)>0$ for all $s\in (0,1)$.

Next, define
\begin{equation*}
    \alpha_1:= \min_{z\in \mathrm{cl}({M_2\setminus M_1})} V_{\infty}(z),\quad \alpha_2:= \max_{z\in \mathrm{cl}({M_2\setminus M_1})} V_{q}(z) 
\end{equation*}
and set $\varepsilon:=\alpha_1/(2\alpha_2)$. 

Now we verify that $V_M$ is a valid, smooth Morse Lyapunov function. 

\begin{enumerate}[(i)]
    \item $V_M^{-1}(0)=0$ as $\varepsilon>0$ and both $V_q$ and $V_{\infty}$ are positive definite and only $0$ at $0$.
    \item $\langle \nabla V_M(x),f(x) \rangle < 0$ for all $x\in \mathbb{R}^n\setminus \{0\}$ due to the following set of arguments. The only non-trivial domain is $\mathrm{cl}(M_2\setminus M_1)$. Then, as $V_{\infty}$ and $V_q$ are both Lyapunov functions for $f$, see that 
    \begin{align*}
        \langle (1-\beta(x))\varepsilon \nabla V_q(x) + \beta(x) \nabla V_{\infty}(x), f(x) \rangle < 0 
    \end{align*}
    on $\mathrm{cl}(M_2\setminus M_1)$. We are left with the term
    \begin{align*}
        \langle(V_{\infty}(x)-\varepsilon V_q(x))\nabla \beta(x),f(x) \rangle.
    \end{align*}
    Note that by our choice of $\varepsilon$, we have that $(V_{\infty}(x)-\varepsilon V_q(x))> 0$ on $\mathrm{cl}(M_2\setminus M_1)$. By our construction of $\beta$, we have that $\nabla \beta(x)$ vanishes only on the boundary of $\mathrm{cl}(M_2\setminus M_1)$. Moreover, on the interior of $\mathrm{cl}(M_2\setminus M_1)$, we have that $\nabla \beta(x) = \gamma(x) \nabla V_{\infty}(x)$ for some smooth strictly positive function $\gamma$. Putting it all together, we find that $ \langle(V_{\infty}(x)-\varepsilon V_q(x))\nabla \beta(x),f(x) \rangle\leq 0$ on $\mathrm{cl}(M_2\setminus M_1)$, hence, $\langle \nabla V_M(x),f(x) \rangle < 0$ on $\mathrm{cl}(M_2\setminus M_1)$. This line of arguments also shows that $V_M$ is smooth. 
    
    \item At last, $V_M$ is Morse as the only critical point of $V_M$ is $0$ and $\nabla^2 V_M(0)=\varepsilon \nabla^2 V_q(0)\succ 0$.
\end{enumerate}
\end{proof}

Observe that the proof of Theorem~\ref{thm:Morse:exists:hyperbolic} already shows that, in the context of $\mathcal{HS}^r_0(\mathbb{R}^n)$, given a Lyapunov function $V$ and some $\varepsilon>0$, one can always find a Morse Lyapunov $V_{M,\varepsilon}$ that is $\varepsilon$-close to $V$, in the $C^1$-topology.

We remark that an alternative proof of Theorem~\ref{thm:Morse:exists:hyperbolic} does not directly follow from employing \cite[Thm. 2.3]{ref:lan2013linearization}, as it is unclear if their change of coordinates can be constructed to be smooth (they work with $C^1$-diffeomorphisms). 

\begin{remark}[Non-trivial basins and smooth manifolds]
\label{rem:basins}
\begin{upshape}
    Let $M$ be a smooth $n$-dimensional manifold and let $p^*\in M$ be asymptotically stable under some smooth vector field $F$. Let $B^*$ be the corresponding basin of attraction. In this case, $B^*$ is diffeomorphic to $\mathbb{R}^n$~\cite[Thm. 2.2]{ref:wilson1967structure}. Hence, if this $p^*$ is some hyperbolic equilibrium point under $F$, then by Theorem~\ref{thm:Morse:exists:hyperbolic}, its stability can be certified via a Morse Lyapunov function. Note, this works as those properties (being hyperbolic, being Morse) are coordinate invariant. 
\end{upshape}
\exampleEnd
\end{remark}

Now, by combining Proposition~\ref{prop:1} and Theorem~\ref{thm:Morse:exists:hyperbolic}, one obtains an alternative proof for some results contained in \cite[App. A]{ref:kvalheim2025differential}.

\begin{corollary}[$\mathcal{HS}_0^r(\mathbb{R}^n)$ is path-connected]
\label{cor:HS:path:connected}
The space $\mathcal{HS}^r_0(\mathbb{R}^n)$, for any $r\in \mathbb{N}_{\geq 1}\cup \{\infty\}$ and $n\in \mathbb{N}_{\geq 1}$, is path-connected. 
\end{corollary}
\begin{proof}
    Following the proof of Proposition~\ref{prop:1}, as $f=f_0\in \mathcal{HS}_0^r(\mathbb{R}^n)$, we have that $f_s=(\psi_s)_* f\in \mathcal{HS}_0^r(\mathbb{R}^n)$ for all $s\in [0,1]$, as $\psi_s$ is a diffeomorphism (guaranteed to exist by Theorem~\ref{thm:Morse:exists:hyperbolic}). At last, since $f_1\in \mathcal{HS}_0^r(\mathbb{R}^n)$, $(1-\tau) Df_1(0)$ is Hurwitz for all $\tau\in [0,1)$, but then, as the path $[0,1]\ni \tau \mapsto (1-\tau) Df_1(0) - \tau I_n$ is Hurwitz, the straight-line homotopy from $f_1$ to $-\nabla V_1/2$ is contained in $\mathcal{HS}_0^r(\mathbb{R}^n)$ as well (\textit{i.e.}, recall that $\nabla V_1(y)/2=y$).
\end{proof}
Another corollary might be of interest with the learning (\textit{e.g.}, via Neural ODEs) of stability certificates in mind. Note, the statement is global. 
\begin{corollary}[Trivial Lyapunov functions]
\label{cor:quad:Lyap}
    Let $f\in \mathcal{HS}^r_0(\mathbb{R}^n)$, for some $r\in \mathbb{N}_{\geq 1}\cup \{\infty\}$ and $n\in \mathbb{N}_{\geq 1}$, then, there is a diffeomorphism $\varphi$ such that 0 being GAS under $\varphi_*f$ can be certified through the quadratic Lyapunov function $y\mapsto \| y\|_2^2$.   
\end{corollary}
\begin{proof}
    As $f\in \mathcal{HS}^r_0(\mathbb{R}^n)$, there is a Morse Lyapunov function $V$ to certify that $0$ is GAS, but then there is diffeomorphism $\varphi$ such that $V(\varphi^{-1}(y))=\|y\|_2^2$. 
\end{proof}
Note, Corollary~\ref{cor:quad:Lyap} is not stating that the vector field itself becomes trivial \textit{cf.}~\cite[p. 156]{ref:Robinson}. It is interesting to remark that even without hyperbolicity, there are global coordinate transformations that do turn the flow under $f$ into the trivial flow (\textit{i.e.}, as generated by $\dot{x}=-x$), however, those transformations are generally not diffeomorphisms \cite{ref:kvalheim2025global}.

\noindent With Proposition~\ref{prop:1} and Theorem~\ref{thm:Morse:exists:hyperbolic} in mind, the remaining question becomes:
\\\\
\textbf{Question:}``\textit{Can Lyapunov functions (with respect to some non-hyperbolic equilibrium point $x^*$ being GAS on $\mathbb{R}^n$) always be chosen to be Morse?}"
\\\\
The case of $n=5$ is of special interest. The following proposition answers this question negatively, via an explicit counterexample. This also means that one cannot prove SPC4 via Proposition~\ref{prop:1} $(ii)$. Similarly, this means that path-connectedness of $\mathcal{S}^r_0(\mathbb{R}^n)$ is weaker than always having access to Morse Lyapunov functions (as $\mathcal{S}^r_0(\mathbb{R}^n)$ is already known to be path-connected for $n\neq 4,5$ \cite[Cor. 7]{ref:kvalheim2025differential}).  

\begin{figure}
    \centering
    \includegraphics[width=0.5\linewidth]{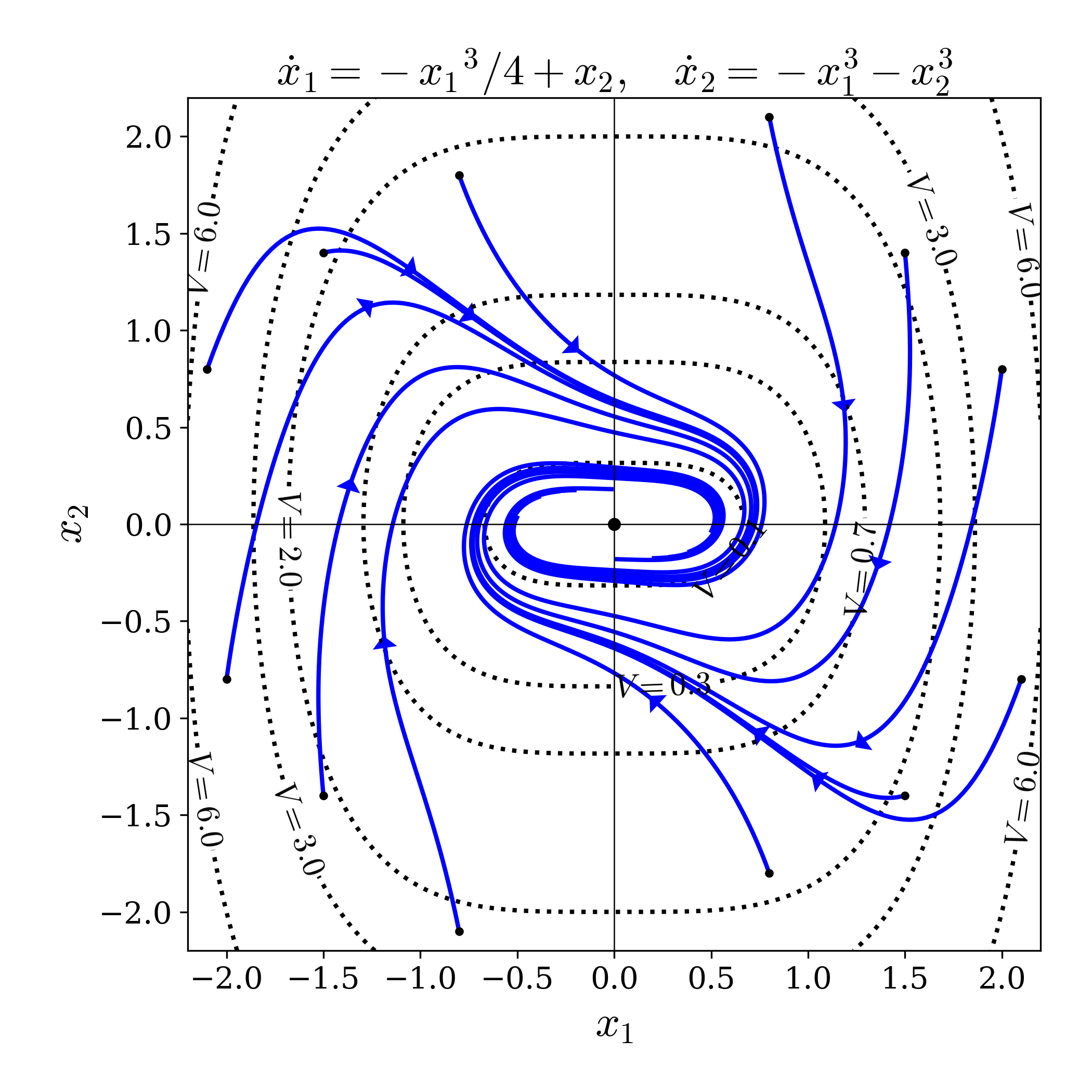}
    \caption{Phase portrait of~\eqref{equ:planar:counter} for $\varepsilon=1/4$. Figure made using Python.}
    \label{fig:placeholder}
\end{figure}

\begin{proposition}[Morse Lyapunov functions do not always exist]
\label{prop:counter}
    For any $n\geq 2$ and $r\in \mathbb{N}_{\geq 1}\cup \{+\infty\}$ there exists a $f\in \mathcal{S}^r_0(\mathbb{R}^n)\setminus \mathcal{HS}_0^r(\mathbb{R}^n)$ such that $0$ being GAS cannot be asserted via a Morse Lyapunov function. 
\end{proposition}
\begin{proof}
    Consider the planar dynamical system $\dot{x}=f(x)$ as given by 
    \begin{equation}
        \label{equ:planar:counter}
            \begin{aligned}
        \dot{x}_1 =& -\varepsilon x_1^3+x_2\\
        \dot{x}_2 =& -x_1^3 - x_2^3,
    \end{aligned}
    \end{equation}
    for some $\varepsilon> 0$. See that $0\in \mathbb{R}^2$ is GAS, as one may check via the smooth Lyapunov function $\mathbb{R}^2 \ni (x_1,x_2)\mapsto {V}(x_1,x_2)=:x_1^4/4+x_2^2/2$ (for $\varepsilon=0$, the same is true, but one employs the Krasovskii-LaSalle invariance principle, \textit{e.g.}, see~\cite[Thm. 5.24]{Sastry}), \textit{e.g.}, see Figure~\ref{fig:placeholder}. Observe that $A:=Df(0)$ is nilpotent and in particular non-hyperbolic. Moreover, $V$ is not a Morse function. 

Suppose now that one could find a Morse Lyapunov function $V_M$, let 
    \begin{align*}
        H = \begin{pmatrix}
a & b \\
b & c
        \end{pmatrix},
    \end{align*}
    denote $\nabla^2 V_M(0)$. It follows that $a,c>0$ and $ac-b^2>0$.  
    
Then see that for any $x\in \mathbb{R}^2$, one has $x^\mathsf{T} HA x =a x_1 x_2 + b x_2^2$. There is always a pair $(x_1,x_2)$ such that this product is strictly positive, for instance, fix $x_2=1$, then $x_1> |b|/a$ suffices.

Next, since both $f$ and $V_M$ are smooth, it follows that $f(x) = Ax + \mathrm{h.o.t.}$ ($\mathrm{h.o.t.}$ denotes \textit{higher order terms}) and $V_M(x) = \frac{1}{2}x^\mathsf{T} H x + \mathrm{h.o.t}$. In particular, $\langle \nabla V_M(x), f(x) \rangle = x^\mathsf{T} HA x + \mathrm{h.o.t}$. Now pick some $\widetilde{x}$ such that $\widetilde{x}^\mathsf{T} HA \widetilde{x}>0$ and see that 
\begin{align*}
    \frac{1}{\tau^2} \langle \nabla V_M(\tau \widetilde{x} ),f(\tau \widetilde{x} )\rangle = \widetilde{x}^\mathsf{T} HA \widetilde{x} + O(\tau)
\end{align*}
is positive for sufficiently small $\tau$. This means that $V_M$ was no Lyapunov function after all. Examples beyond $n=2$ are found by appending \eqref{equ:planar:counter} with $\dot{z}=-z$, for $z\in \mathbb{R}^{n-2}$. In that case the obstruction follows by considering states of the form $(\widetilde{x},0)$.  
\end{proof}

\begin{remark}[The case $n=1$]
\begin{upshape}
    When $f\in \mathcal{S}^r_0(\mathbb{R}^n)$,  for $n = 1$ and $r\in \mathbb{N}_{\geq 1}\cup \{+\infty\}$, the trivial (quadratic) Morse Lyapunov function, \textit{i.e.}, $x\mapsto V_M(x)=x^2$, always suffices as only signs matter, \textit{e.g.}, recall $\dot{x}=-x^3$. See also~\cite{ref:sontag1980remarks}. 
\end{upshape}
\exampleEnd
\end{remark}

Characterizing precisely when Morse Lyapunov functions exist remains an open problem. We provide a few pointers. First, beyond $n>1$, a lack of hyperbolicity is also no immediate obstruction.

\begin{figure}
    \centering
    \includegraphics[width=0.5\linewidth]{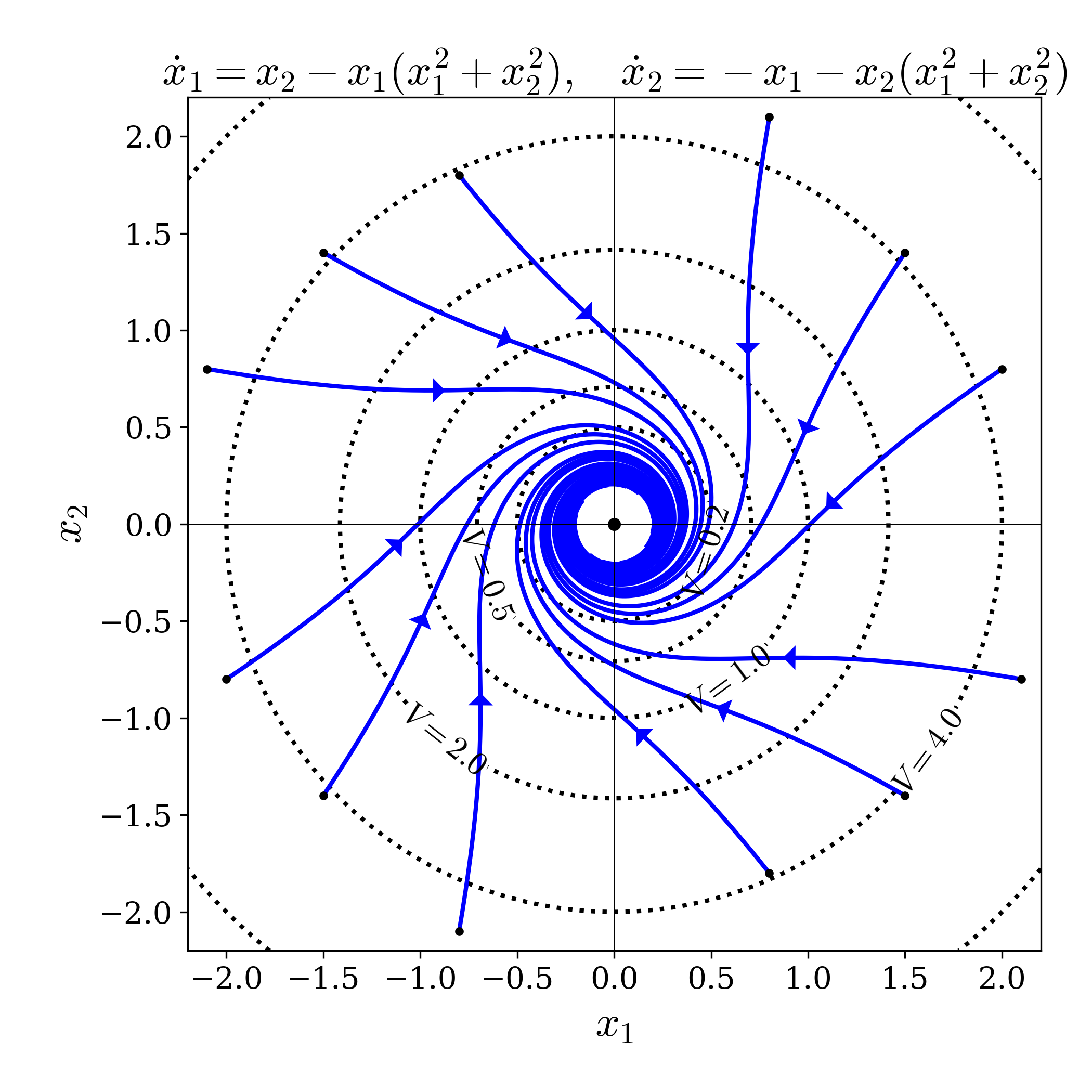}
    \caption{Phase portrait of~\eqref{equ:planar:nonhyperbolic:with:Morse}. Figure made using Python.}
    \label{fig:nonhyperbolic:Morse}
\end{figure}

\begin{example}[Non-hyperbolic planar system with a Morse Lyapunov function]
\label{ex:nonhyperbolic:but:Morse}
\begin{upshape}
         Consider the planar dynamical system $\dot{x}=f(x)$ as given by 
    \begin{equation}
        \label{equ:planar:nonhyperbolic:with:Morse}
            \begin{aligned}
        \dot{x}_1 =& x_2 - x_1(x_1^2+x_2^2),\\
        \dot{x}_2 =& -x_1-x_2(x_1^2+x_2^2),
    \end{aligned}
    \end{equation}
    this is a standard example in the context of center manifolds, \textit{e.g.}, see~\cite[\S 7.6]{Sastry}. Observe that $0\in \mathbb{R}^2$ is the unique equilibrium point and since $\mathrm{spec}(Df(0))=\{\pm i\}$, it is a non-hyperbolic equilibrium point. However, $(x_1,x_2)\mapsto V_M(x_1,x_2)=x_1^2+x_2^2$ checks out as a Morse Lyapunov function, see Figure~\ref{fig:nonhyperbolic:Morse}. 
\end{upshape}
    \exampleEnd
\end{example}

Towards a complete characterization of when Morse Lyapunov functions exist, we employ the following lemma. 
\begin{lemma}[Marginal stability and Lyapunov equations]
\label{lem:lyapunov:equ}
    For $A\in \mathbb{R}^{n\times n}$ there exists a $H\succ 0$ such that $HA+A^{\mathsf{T}}H\preceq 0$ if and only if $A$ has $\Re(\mathrm{spec}(A))\leq 0$ and all eigenvalues $\lambda_i$ of $A$, such that $\Re(\lambda_i)=0$, are semisimple. 
\end{lemma}
\begin{proof}
    We believe this is a classical result (\textit{e.g.}, see \cite[Exercise 5.6(b)]{ref:Astrom2008feedback}), yet we provide a proof to be self-contained. 

    We start by showing that if $A$ satisfies the assumptions, then such a $H$ exists. Decompose $\mathbb{R}^n=E^s \oplus E^c$. On $E^s$ (the stable subspace), we have that $A^s:=A|_{E^s}$ is Hurwitz, and thus there is a matrix $H^s$ such that $H^s A^s + (A^s)^{\mathsf{T}}H^s \prec 0$. On the center subspace $E^c$ we know that the eigenvalues are semisimple, meaning that we can diagonalize $A^c := A|_{E^c}$, where the blocks are either $0$ or of the form 
\begin{align*}
    A^{\omega} = \begin{pmatrix}
        0 & \omega \\ - \omega & 0
    \end{pmatrix}.
\end{align*}
See that if we select $H^{\omega}=I_2$, then $H^{\omega} A^{\omega}+ (A^{\omega})^{\mathsf{T}}H^{\omega} = 0$. Hence, if we select $H:=\mathrm{diag}(H^s,I_p)$, for $p=\mathrm{dim}(E^c)$, then, the Lyapunov equation $HA+A^{\mathsf{T}}H\preceq 0$ is satisfied. 

Now, for the other direction, assume that $H$ exists. Let $V(x):= \langle Hx, x\rangle $ and consider $V$ as being a Lyapunov function for $\dot{x}=Ax$. As $\langle Hx, Ax \rangle \leq 0$ by assumption, $V$ is non-increasing along solutions under $A$. This directly implies that $\Re(\mathrm{spec}(A))\leq 0$. Now suppose that $A$ has some eigenvalue $\lambda_i$ such that $\Re(\lambda_i)=0$, yet $\lambda_i$ fails to be semisimple. This means that diagonalization is obstructed and we have a non-trivial Jordan block. This however means that the solution under $A$ is unbounded, hence, a contradiction.    
\end{proof}

Now, finally, we can provide a necessary condition for the existence of a smooth Morse Lyapunov function.

\begin{corollary}[Existence of Morse Lyapunov functions]
\label{cor:no:Morse:exists}
  Let $f\in \mathcal{S}^r_0(\mathbb{R}^n)$, for some $r\in \mathbb{N}_{\geq 1}\cup \{\infty\}$ and $n\in \mathbb{N}_{\geq 1}$. Then, there is a smooth Morse Lyapunov function certifying that $0$ is GAS only if the non-hyperbolic (center) part of $Df(0)$ can be diagonalized. 
\end{corollary}
\begin{proof}
    For if not, $Df(0)$ must have an eigenvalue $\lambda_i$, with $\Re(\lambda_i)=0$, that is not semisimple. Then, by Lemma~\ref{lem:lyapunov:equ} and the proof strategy as in Proposition~\ref{prop:counter}, it follows that no Morse Lyapunov function can exist. 
\end{proof}

Indeed, the example contained in the proof of Proposition~\ref{prop:counter} is such that $Df(0)$ is nilpotent, whereas the system in Example~\ref{ex:nonhyperbolic:but:Morse} does meet the conditions from Corollary~\ref{cor:no:Morse:exists}. 

However, the vector fields from both \cite[Prop. 5.2 (b)]{ref:bacciotti2005liapunov} and \cite{ref:ahmadi2018globallyP} are such that $Df(0)=0\cdot I_2$, hence, they also meet the conditions from Corollary~\ref{cor:no:Morse:exists}. The condition is however only necessary and our theory does not allow for any conclusive statements. As one cannot readily combine Lemma~\ref{lem:lyapunov:equ} and the proof of Theorem~\ref{thm:Morse:exists:hyperbolic} (\textit{i.e.}, as the effect of higher order terms is unclear), we leave a complete characterization of Morse Lyapunov functions for future work. 

%%%%%%%%%%%%%%%%%%%%%%%%%%%%%%%%%%%%%%%%%%%%
\pagestyle{basicstyle}
\addcontentsline{toc}{section}{Bibliography}
\subsection*{Bibliography}
\printbibliography[heading=none]

%%%%%%%%%%%%%%%%%%%%%%%%%%%%%

\end{document}